\documentclass[11pt]{amsart}

\usepackage[utf8]{inputenc}
\usepackage{amsmath,amsxtra,amssymb,latexsym,amscd,amsthm}

\theoremstyle{plain}

\usepackage{csquotes}

\usepackage{tikz-cd}
\usepackage{tkz-euclide}

\usepackage[colorlinks]{hyperref} %pagebackref
\hypersetup{
  pdftitle   = {Equivariant Kuranishi family of complex compact manifolds},
  pdfauthor  = {An-Khuong Doan},
  pdfcreator = {},
  linkcolor  = {black},
  citecolor  = {black},
  urlcolor   = {black}
}

\usepackage{url}

\newtheorem{theorem}{Theorem}[section]

\theoremstyle{definition}

\theoremstyle{remark}

\numberwithin{equation}{section}

\theoremstyle{break}
\newtheorem{defi}{Definition}[section]
\newtheorem{thm}{Theorem}[section]
\newtheorem{prop}{Proposition}[section]
\newtheorem{lem}{Lemma}[section]
\newtheorem{coro}{Corollary}[section]

\theoremstyle{remark}
\newtheorem{rem}{Remark}[section]
\newtheorem*{ackn}{\textbf{Acknowledgements}}

\newtheorem*{thm*}{Theorem}

\newcommand{\Ho}{\operatorname{Hol}}
\newcommand{\Au}{\operatorname{Aut}}
\newcommand{\Lie}{\operatorname{Lie}}
\newcommand{\Diff}{\operatorname{Diff}}
\begin{document}

\title{Equivariant Kuranishi family of complex compact manifolds}

\author{}
%    Address of record for the research reported here
\address{}
%    Current address
\email{ }
%    \thanks will become a 1st page footnote.
\thanks{ }

\author{ An-Khuong DOAN}
\address{An-Khuong DOAN, IMJ-PRG, UMR 7586, Sorbonne Université,  Case 247, 4 place Jussieu, 75252 Paris Cedex 05, France}
\email{an-khuong.doan@imj-prg.fr }
\thanks{ }

%    General info

\date{December 12, 2020.}

\dedicatory{ }


\begin{abstract}
We prove that actions of complex reductive Lie groups on a complex compact manifold are locally extendable to its Kuranishi family. This can be seen as an analogue of Rim's result (see [12]) in the analytic setting.
\end{abstract}

\maketitle
\tableofcontents
\section{Introduction}

Let $X_0$ be an algebraic scheme equipped with an action of an algebraic group $G$, defined over a fixed algebraically closed field $k$. D. S. Rim posed a problem which asks whether we can provide a $G$-action on the formal semi-universal deformation of $X_0$, extending the given $G$-action. If $X_0$ is an affine cone with $\mathbb{G}_m$-action, Pinkham showed that the answer is affirmative in [11]. Later on, Rim generalized this result to the case that  $X_0$ is an affine scheme with at most isolated singularities or a complete algebraic variety over $k$ and that $G$ is a linearly reductive algebraic group (see [12]). For non-reductive groups, this is not the case, in general. A counter-example to this phenomenon can be found in [5], where $X_0$ is the second Hirzebruch surface $\mathbb{F}_2$ and $G$ is its automorphism group. In this paper, we would like to reproduce Rim's result when $X_0$ is a complex compact manifold on which a compact Lie group $G$ acts holomorphically and then try to address the case that $G$ is a complex reductive Lie group (see Corollary 4.1 and Theorem 5.2 below).

A remark should be in order. The main different point here is that in the algebraic setting, the semi-universal deformation of $X_0$ and the extended $G$-actions, constructed by Rim, are just formal. However, in the analytic setting, its semi-universal deformation (often called Kuranishi family) is a true deformation (a convergent deformation). So, an application of Rim's result gives us a $G$-equivariant Kuranishi family whose extended $G$-actions are only formal, i.e. they are formal power series whose convergence is not guaranteed. Actually, this way of using Rim's theorem keeps being repeated several times for example in the proof of Theorem 4.20 in [10] and in the proof of Proposition 7.1 in [7], where the convergence is supposedly needed to carry out. Moreover, an extension of the $G$-action on the Kuranishi space is immediate if the Kuranishi family is locally universal. This follows from the fact that each time we change the central fiber of the locally universal family by a biholomorphism of $X_0$, we obtain another locally universal family of $X_0$ which is canonically isomorphic to the old one. However, in the proof of Lemma 3.4 in [4], the author produced a $G$-action on the base by claiming that there exists a local universal deformation, which is not true in general even for the type of complex compact manifold considered therein. Thus, a very natural wish is to have a convergent $G$-extension. This is one of the motivations for us to write the paper.

Let us now outline the organization of this article. First, we give a general picture of deformations of complex compact manifolds in $\S2$. The most important result on the existence of semi-universal deformation (Kuranishi family) is also included. Next, we attack the problem by giving a useful existence criterion in $\S3$, which turns out to be deduced from an elementary lemma on complex structures of real vector spaces. In $\S4$, we treat the case that $G$ is compact, in advance. The key point here is that in place of imposing an arbitrary Hermitian metric on the holomorphic tangent bundle, we can impose a $G$-invariant one for the sake of the compactness of $G$. In fact, this idea is already contained in Catanese's lecture note (see [3, Lecture III, $\S$7]). However, the author uses it to treat only the case that the actions are required to be trivial on the base. If we take the set of fixed points by the $G$-action in the base constructed in our case then the restriction of our $G$-equivariant family on this set is nothing but Catanese's family. Finally, in $\S5$, we deal with the complex reductive case by means of complexification of compact groups.

 \begin{ackn} This is a part of my Ph.D thesis at Institut de mathém-atiques de Jussieu – Paris Rive Gauche (IMJ-PRG). I would like to profoundly thank my thesis advisor - Prof. Julien Grivaux for many precious discussions, for his enthusiastic instructions, his continuous support, his extremely careful reading and his comments, which help enormously to establish this work. I am warmly grateful to the referee whose work led to a considerable improvement of the paper.
 \end{ackn}
 \section{Deformations of complex compact manifolds}
We first recall some basic definitions in deformation theory of compact complex manifolds. Let $\mathfrak{B}$ be the category of germs of pointed complex space $(B,0)$ (a complex space with a reference point) whose associated reduced complex space is a point and let $X_0$ be a complex compact manifold. An infinitesimal deformation of $X_0$ is a deformation of $X_0$ over a germ of complex space $(B,0)\in \mathfrak{B}$, i.e. a commutative diagram
\begin{center}
\begin{tikzpicture}[every node/.style={midway}]
  \matrix[column sep={8em,between origins}, row sep={3em}] at (0,0) {
    \node(Y){$X_0$} ; & \node(X) {$X$}; \\
    \node(M) {$\cdot$}; & \node (N) {$(B,0)$};\\
  };
  
  \draw[->] (Y) -- (M) node[anchor=east]  {}  ;
  \draw[->] (Y) -- (X) node[anchor=south]  {$i$};
  \draw[->] (X) -- (N) node[anchor=west] {$\pi$};
  \draw[->] (M) -- (N) node[anchor=north] {};.
\end{tikzpicture}
\end{center}
where $\pi:\;X\rightarrow (B,0)$ is a flat proper morphism of complex spaces. For simplicity, we denote such a deformation by $\pi$: $X\rightarrow (B,0)$ (or sometimes just $X/B$). If  $\pi$: $X\rightarrow (B,0)$ and  $\pi'$: $X'\rightarrow (B',0)$ are two infinitesimal deformations of $X_0$, a morphism of infinitesimal deformations is a pair $(\Phi,\phi)$ of two morphisms of complex spaces $\Phi:\;X\rightarrow X'$ and $\phi:\;(B,0) \rightarrow (B',0)$  such that the following diagram commutes

\begin{center}
\begin{tikzpicture}[every node/.style={midway}]
  \matrix[column sep={4em,between origins}, row sep={1em}] at (0,0) {
  \node(A){} ; &\node(B){$X$} ; &  \node(C){} ; & \node(D) {$X'$}; \\
  \node(E){$X_0$} ; &\node(F){} ; &  \node(G){} ; & \node(H) {}; \\
  \node(I){} ; &\node(K){$(B,0)$} ; &  \node(L){} ; & \node(M) {$(B',0).$}; \\
  \node(N){$.$} ; &\node(O){} ; &  \node(P){} ; & \node(Q) {}; \\
  };
  
  \draw[->] (B) -- (D) node[anchor=south]  {$\Phi$}  ;
  \draw[->] (B) -- (K) node[anchor=west]  {$\pi$};
  \draw[->] (D) -- (M) node[anchor=west] {$\pi'$};
  \draw[->] (K) -- (M) node[anchor=south] {$\phi$};
  \draw[->] (E) -- (B) node[anchor=south] {$i$};
  \draw[->] (E) -- (D) node[anchor=north] {$i'$};
   \draw[->] (N) -- (K) node[anchor=south] {};
  \draw[->] (N) -- (M) node[anchor=north] {};
  \draw[->] (E) -- (N) node[anchor=north] {};.
\end{tikzpicture}
\end{center}

Kuranishi proves the existence of a semi-universal deformation $\pi$: $X\rightarrow (S,0)$, called Kuranishi family, which contains all the information of small deformations of $X_0$ (cf. [8] or [9]). Semi-universality here means that any other deformation $\rho$: $Y \rightarrow (T,0)$ of $X_0$ is defined by the pullback of the Kuranishi family under a holomorphic map from $(T,0)$ to $(S,0)$, whose differential at the reference point is unique.

Next, let us take a moment to recall the definition of group actions on complex spaces. For the sake of completeness, we recall first that a mapping $\alpha$ from a real analytic (resp. complex) manifold $W$ to a Fréchet space $F$ over $\mathbb{C}$ is called \textit{real analytic} (reps. \textit{holomorphic}) if for each point $w_0\in W$ there exists an open coordinate neighborhood $N_{w_0}$ and a real analytic (resp. holomorphic) coordinate system $t_1,\ldots,t_n $ in $N$ such that $t_i(w_0)=0$ and for all $w\in N$, we have
that $$\alpha(w)=\sum a_{i_1,\ldots,i_n}t_1^{i_1}(w)\ldots t_n^{i_n}(w) $$ where $a_{i_1,\ldots,i_n} \in F$ and the convergence is absolute with respect to any continuous semi-norm on $F$. Furthermore, by a $C^p$-map, we insinuate a $p$-times continuously differentiable function. Let $G$ be a real (resp. complex) Lie group and $X$ a complex space. A $G$-action on $X$ is given by a group homomorphism $\Phi:\; G \rightarrow \Au(X)$, where $\Au(X)$ is the group of biholomorphisms of $X$. 
\begin{defi} The $G$-action determined by $\Phi$ is said to be real analytic (resp. holomorphic) if for each open relatively compact $U \Subset X$ and for each open $V\subset X$, the following conditions are satisfied
\begin{enumerate}
\item[(i)]$W:=W_{\overline{U},V}:=\lbrace g\in G \mid g\cdot \overline{U}\subset V \rbrace $ is open in $G$,
\item[(ii)]the map \begin{align*}
*:W&\rightarrow \mathcal{O}(U)\\
g &\mapsto f\circ g\mid_U
\end{align*} is real analytic (resp. holomorphic) for all $f\in \mathcal{O}(V)$ ,
\end{enumerate}
where $\overline{U}$ is the closure of $U$ and $\mathcal{O}(P)$ is the set of holomorphic functions on $P$ for any open subset $P$ of $X$ ($\mathcal{O}(P)$ is equipped with the canonical Fréchet topology).
\end{defi}
To end this section, we introduce a very interesting kind of deformations-the kind of $G$-equivariant ones, which is of central interest of the article. As before, let $X_0$ be a complex compact manifold equipped with a real analytic (resp. holomorphic) $G$-action.
\begin{defi}
 A real analytic (resp. holomorphic) $G$-equivariant deformation of $X_0$ is a usual deformation of $X_0$ $\pi$: $X\rightarrow B$ equipped with a real analytic (resp. holomorphic) $G$-action on $X$ extending the given (resp. holomorphic) $G$-action on $X_0$ and a real analytic (resp. holomorphic) $G$-action on $B$ in a way that $\pi$ is a $G$-equivariant map with respect to these actions. We call these extended actions a real analytic (resp. holomorphic) $G$-equivariant structure on $\pi$: $X\rightarrow B$.
\end{defi}
 Therefore, we can rephrase our objective as finding a real analytic (resp. holomorphic) $G$-equivariant semi-universal deformation of a given compact complex manifold with a real analytic (resp. holomorphic) $G$-action.
\begin{rem}
For simplicity, by $G$-actions (resp. $G$-equivariant deformations), we really mean real analytic $G$-actions (resp. real analytic $G$-equivariant deformations).
\end{rem}

\section{A sufficient condition for the existence of equivariant structure} In this section, we give a criterion for a complex compact manifold $X_0$ with a $G$-action to have a $G$-equivariant semi-universal deformation. From now on, by complex compact manifold, we really mean a complex compact connected manifold.  First, we recall a technical result concerning the holomorphicity of real analytic functions defined on complex spaces (cf. [8, Proposition 2.1]).
\begin{prop} If $V$ is a complex space and $v$ is a point of $V$, there exists an integer $\alpha$ satisfying the following condition: If $f:\; V\rightarrow V'$ is a $C^\alpha$-map, where $V'$ is another complex space, such that $f$ is holomorphic at each non-singular point of $V$ then there is an open neighborhood $\overline{V}$ of $v$ in $V$ such that the restriction of $f$ on $\overline{V}$ is holomorphic.
\end{prop}
Denote by $\Diff(\underline{X_0})$ the group of diffeomorphisms of $\underline{X_0}$ where  $\underline{X_0}$ is the underlying differentiable manifold of $X_0$. For $S$ a complex space, a map $\gamma: \; S \rightarrow \Diff(\underline{X_0})$ is said to be of class $C^k$ when the map  \begin{align*}
\Gamma :\underline{X_0}\times S &\rightarrow \underline{X_0}\\
(p,s) &\mapsto \gamma(s)(p)
\end{align*} is of class $C^k$. If this is indeed the case, then for each $s_0\in S$ the map
\begin{align*}
\Gamma_{s_0} :\underline{X_0}\times S &\rightarrow \underline{X_0}\\
(p,s) &\mapsto \gamma(s)\circ (\gamma(s_0))^{-1}(p)
\end{align*} is a $C^k$-family of deformations of the identity map of $\underline{X_0}$ with a parameter in $(S,s_0)$. In particular, for each $p \in \underline{X_0}$, we obtain a $C^k$-map
\begin{align*}
\Gamma_{s_0,p} : S &\rightarrow \underline{X_0}\\
s &\mapsto \gamma(s)\circ (\gamma(s_0))^{-1}(p).
\end{align*} Therefore, if we suppose further that $s_0$ is a non-singular point then each $L \in T_{s_0}^{\text{Zar}}S$ will give rise to a vector $d(\Gamma_{s_0,p})_{s_0}(L)$, in $T_p{\underline{X_0}}$, where  $d(\Gamma_{s_0,p})_{s_0}$ is the differential of $\Gamma_{s_0,p}$ at $s_0$. Thus, the map 
\begin{align*}
\underline{X_0} &\rightarrow T\underline{X_0}\\
p &\mapsto d(\Gamma_{s_0,p})_{s_0}(L)
\end{align*} defines a $C^k$-vector field on $\underline{X_0}$, which we shall denote by $L\sharp^{s_0} \gamma $.

Finally, before 
stating the main result, given a complex compact manifold $X_0$, let us bring back a celebrated characterization of its deformations and in particular of its semi-universal deformation (see [8, Theorem 8.1]).

\begin{theorem}

A deformation of $X_0$ is entirely encoded by a real analytic map $\phi: S \rightarrow A^{0,1}(\Theta)$ which varies holomorphically in $S$ such that 
\begin{enumerate}
\item[$(i)$] $\phi(0)=0$,
\item[$(ii)$] $\overline{\partial}\phi(s)-\frac{1}{2}[\phi(s),\phi(s)]=0$ for all $s \in S$,
\end{enumerate}
where $A^{0,1}(\Theta)$ is the space of $(0,1)$-forms with values in the holomorphic tangent bundle $\Theta$ of $X_0$ and  $S$ is a complex space with a reference point $0$. Moreover, this deformation is semi-universal if and only if 
\begin{enumerate}
\item[$(iii)$] The Kodaira-Spencer map induced by $\phi$ is an isomorphism,

\item[$(iv)$] We can find an open neighborhood $S'$ of $0$ in $S$ such that the following conditions hold true: for any complex space $B$ and for any real analytic map $\psi:\;  B \rightarrow  A^{0,1}(\Theta)$, which varies holomorphically in $B$, such that $\psi(b_1)=\phi(s_1)$ for a point $(b_1,s_1)\in B\times S'$, we can find a neighborhood $B'$ of $b_1$, a holomorphic map $\tau: \; (B',b_1)\rightarrow (S',s_1)$ and a $C^{\alpha}$-map $\gamma:\; B' \rightarrow \Diff(\underline{X_0})$ such that 
\begin{enumerate}
\item[$(a)$] $\phi(\tau(b))=\psi(b)\circ \gamma(b)$ for all $b\in B'$. Here, $\psi(b)\circ \gamma(b)$ is the complex structure induced by the complex structure $\psi(b)$ and the diffeomorphism $\gamma(b)$,
\item[$(b)$] For each regular point $b \in B'$ and for all $L \in T_b^{0,1} B \subset T_b^{\text{Zar}}B=  T_b^{1,0} B \oplus T_b^{0,1} B$, we have that $L\sharp^b \gamma^{-1} +\phi(\tau(b))\circ \overline{\overline{L}\sharp^b(\gamma^{-1})} =0$ where $\alpha$ is the integer in Proposition 3.1 for $(  \mathbb{C}^{ \dim_{\mathbb{C} } X_0} \times B, 0\times b_1   )$ and $\gamma^{-1}$ is the map $B' \rightarrow \Diff(\underline{X_0})$ which to $b \in B'$, associates $(\gamma(b))^{-1}$.
\end{enumerate}
\end{enumerate}

\end{theorem}
Now, coming back to our case where the group action joins the game, we claim the following.
\begin{theorem}
If the map $\phi$ can also be made $G$-equivariant with respect to some $G$-action on $S$ and the $G$-action on $A^{0,1}(\Theta)$, induced by the one on $X_0$, then a $G$-equivariant semi-universal deformation of $X_0$ exists.
\end{theorem}

 In order to prove this, let us introduce a lemma on complex structures of real vector spaces. Let $V$ be a real vector space of even dimension imposed with three different complex structures $J, J_m, J_n$  and $V^{\mathbb{C}}$ be its complexification then we have three complex vector spaces $(V,J),(V,J_m),(V,J_n)$ and  decompositions  $$V^{\mathbb{C}}=V_J^{1,0}\oplus V_J^{0,1},V^{\mathbb{C}}=V_{J_m}^{1,0}\oplus V_{J_m}^{0,1}, \text{ and }V^{\mathbb{C}}=V_{J_n}^{1,0}\oplus V_{J_n}^{0,1} $$ where $V_.^{1,0}$ and $V_.^{0,1}$ are eigenspaces attached to the eigenvalues $i$ and $-i$, respectively. Let $\pi^{1,0}:$ $V^{\mathbb{C}} \rightarrow V_J^{1,0}$ and  $\pi^{0,1}:$ $V^{\mathbb{C}} \rightarrow V_J^{0,1}$ be the canonical projections.

Now, suppose that the restrictions of $\pi^{0,1}$ on $V_{J_m}^{0,1}$ and on  $V_{J_n}^{0,1}$ are isomorphisms. Define $m,n: V_J^{0,1} \rightarrow V_J^{1,0}$ by
$m=\pi^{1,0}\circ(\pi^{0,1}\mid_{V_{J_m}^{0,1}})^{-1}$ and $n=\pi^{1,0}\circ(\pi^{0,1}\mid_{V_{J_n}^{0,1}})^{-1}$.  It is well-known that $$V_{J_m}^{0,1}=\left \{ u+m(u)\mid u\in V_{J}^{0,1} \right \} \text{ and }V_{J_n}^{0,1}=\left \{ u+n(u)\mid u\in V_{J}^{0,1} \right \}.$$

\begin{lem}
Let $\varphi$: $V \rightarrow V$ be an $\mathbb{R}$-linear map such that its complexification $\varphi^{\mathbb{C}}$  is a $\mathbb{C}$-linear map from $(V,J)$ to $(V,J)$. Then $\varphi$ is $\mathbb{C}$-linear as a map from $(V,J_m)  $ to $ (V,J_n)$ if $\varphi^{\mathbb{C}}\circ m=n\circ \varphi^{\mathbb{C}}$.
\end{lem}
\begin{proof}
We claim that $\varphi^{\mathbb{C}}(V_{J_m}^{0,1})\subseteq V_{J_n}^{0,1}$. Indeed, let $v\in V_{J_m}^{0,1} $ then $v= u+m(u)$ for some $u\in V_{J}^{0,1}$. So, \begin{align*}
\varphi^{\mathbb{C}}(v) &= \varphi^{\mathbb{C}}( u+m(u))\\ 
 &=\varphi^{\mathbb{C}}( u)+\varphi^{\mathbb{C}}\circ m(u)\\ 
 &= \varphi^{\mathbb{C}}( u)+n\circ \varphi^{\mathbb{C}}(u).
\end{align*}
Moreover, since $\varphi^{\mathbb{C}}$ is a $\mathbb{C}$-linear map from $(V,J)$ to $(V,J)$ then \begin{align*}
J\varphi^{\mathbb{C}}(u) &=\varphi^{\mathbb{C}} J(u) \\ 
 &= \varphi^{\mathbb{C}}(iu) \text{ since }u\in V_{J}^{0,1}\\ 
 &= i\varphi^{\mathbb{C}}(u),
\end{align*} which implies that $\varphi^{\mathbb{C}}(v) \in V_{J}^{0,1} $. Hence, $\varphi^{\mathbb{C}}( u)+n\circ \varphi(u^{\mathbb{C}}) \in V_{J_n}^{0,1}$ then so is $\varphi(v)$, which proves the claim.

Now, let $v \in V_{J_m}^{0,1}$, then \begin{align*}
J_n\varphi^{\mathbb{C}}(v) &=-i\varphi^{\mathbb{C}}(v)\text{ by the claim} ,\\ 
 &= \varphi^{\mathbb{C}}(-iv)\\ 
 &= \varphi^{\mathbb{C}}J_m(v).
\end{align*}
Making use of the linear complex conjugation, we also get that $$J_n\varphi^{\mathbb{C}}(v)=\varphi^{\mathbb{C}}J_m(v)$$ for all $v \in V_{J_m}^{1,0}$. This ends the proof.
\end{proof}

Finally, it is the time for us to prove Theorem 3.2.
\begin{proof}[Proof of Theorem 3.2] First of all, by the discussion at the very beginning of this section, we have a semi-universal deformation $\pi: X \rightarrow S$ of $X_0$, associated to $\phi$. Let $\underline{X_0}$ be the underlying differentiable manifold of $X_0$. By [3, Theorem 4.5] after shrinking $S$ if necessary, there exists a real analytic diffeomorphism $\gamma: \; \underline{X_0}\times S \rightarrow X$ with $\pi\circ \gamma$ being the projection on the second factor of $\underline{X_0}\times S $, and such that $\gamma$ is holomorphic in the second set of variables. Thus, for a point $(x,s) \in \underline{X_0}\times S$, we have a decomposition of  the tangent space 
$$T_xX\oplus T_s^{\text{Zar}}S \cong T_{\gamma(x,s)}^{\text{Zar}}X   .$$
We claim that $\pi:$ $X \rightarrow S$ carries a $G$-equivariant structure. Indeed, for $g\in G$ and $(x,s) \in \underline{X_0}\times S $, define $$g.(x,s)=(g.x,g.s)$$ in which we think of $g$ as just a diffeomorphism of $\underline{X_0}$. This gives clearly an action of $G$ on $\underline{X_0}\times S $. We shall prove that in fact if we think of $X$ as $\underline{X_0}\times S$ with the complex structure $\phi(-)$ then $G$ acts on $X$ by biholomorphisms. This is equivalent to showing that the differential of $g$ at the point $(x,s)$ 
$$dg_{(x,s)}: \;T_{(x,s)}^{\text{Zar}}X=T_x\underline{X_0}\oplus T_s^{\text{Zar}} \rightarrow  T_{g.(x,s)}^{\text{Zar}}X=T_{g.x}\underline{X_0}\oplus T_{gs}^{\text{Zar}} $$ is $\mathbb{C}$-linear with respect to the complex structure induced by $\phi$ on the tangent space $T_{(x,s)}^{\text{Zar}}X$. Since $dg_{(x,s)}=(dg_x,dg_s)$ is a diagonal map and $g$ 
acts holomorphically on $S$. Then it is sufficient to check that $$dg_x:\;(T_x\underline{X_0},J_{(x,s)}) \rightarrow (T_{gx}\underline{X_0},J_{(gx,gs)})$$ is $\mathbb{C}$-complex linear where $J_{(x,s)}$ and $J_{(gx,gs)}$ are complex structures induced by maps $\phi(s)_x: $ $T_x^{0,1}\underline{X_0} \rightarrow T_x^{1,0}\underline{X_0}  $ and $\phi(gs)_{gx}:$ $T_{gx}^{0,1}\underline{X_0} \rightarrow T_{gx}^{1,0}\underline{X_0}$, respectively.
On the other hand, as $\phi$ is $G$-equivariant then we have
$$g\phi(s)=\phi(gs)$$ for any $s \in S$. This is equivalent to 
$$dg\phi(s)dg^{-1}=\phi(gs),$$ by definition of the action of a diffeomorphism $g$ on a complex structure $\phi(s)$. Thus, for each $x\in \underline{X_0}$, $$dg_x\phi(s)_x=\phi(gs)_{gx}dg_x.$$
Making use of Lemma 3.1 for $m=\phi(s)$, $n=\phi(gs)$ and $\varphi=dg_x$, we deduce that $dg_x$ is $\mathbb{C}$-complex linear so that $g$ is in fact holomorphic. Thus, we have just extended the $G$-action on the central fiber $X_0$ to a $G$-action the total space $X$. This action together with the given $G$-action on $S$ makes $\pi$ $G$-equivariant, which completes the proof. 
 \end{proof}

\section{The case that $G$ is a compact Lie group}
We treat compact group actions first. Let $X_0$ be an $n$-dimensional complex compact manifold equipped with a real analytic $K$-action, where $K$ is a compact real Lie group. The main result of this section is the following.
\begin{thm}
There exists a complex space $(S,0)$ and a real analytic map $\phi: (S,0) \rightarrow A^{0,1}(\Theta)$ which varies holomorphically in $S$ such that the conditions $(i),(ii),(iii)$ and $(iv)$, listed in Theorem 3.1, are fulfilled. Furthermore,  $\phi$ is $K$-equivariant with respect to some $K$-action on $S$ and the $K$-action on $A^{0,1}(\Theta)$, induced by the one on $X_0$.
\end{thm}
\begin{coro} Let $X_0$ be a complex compact manifold $X_0$ with a $K$-action, where $K$ is a compact real Lie group. Then there exists a $K$-equivariant semi-universal deformation of $X_0$. 
\end{coro}
\begin{proof} It follows immediately from Theorem 4.1 above and Theorem 3.2.
\end{proof}

In order to prove Theorem 4.1, we shall follow Kuranishi's method in $[8]$ with some appropriate modification. First of all, note that we have a natural linear $K$-action on $A^{0,1}(\Theta)$ and then on $H^1(X_0,\Theta)$. Moreover, since $K$ is compact, instead of imposing an arbitrary hermitian metric on $\Theta$ as Kuranishi did, we can impose a $K$-invariant Hermitian metric $\left \langle \cdot,\cdot \right \rangle$ on $\Theta$ by means of Weyl's trick. Therefore, we have a $K$-invariant metric on $A^{0,1}(\Theta)$. As usual, we find the formal adjoint $\overline{\partial}^*$ of $\overline{\partial}$. Since $K$ acts on $X_0$ by biholomorphisms then the operator $\overline{\partial}$ is $K$-equivariant. By the adjoint property together with the fact that the imposed metric is $K$-invariant, we also have that $\overline{\partial}^*$ is $K$-equivariant. Hence, so is the Laplacian 
$\square:= \overline{\partial}^*\overline{\partial}+\overline{\partial}\overline{\partial}^*$. In	addition, it is well-known that $\square$ is an elliptic operator of second order. As a matter of fact, Hodge theory provides us a famous orthogonal decomposition.
\begin{equation} A^{0,1}(\Theta)=\mathcal{H}^{0,1}\bigoplus \square A^{0,1}(\Theta) 
\end{equation} and two linear operators:
\begin{enumerate}
\item[(a)] The Green operator $G:$ $A^{0,1}(\Theta)\rightarrow\square A^{0,1}(\Theta) $,
\item[(b)] The harmonic projection operator $H:$ $A^{0,1}(\Theta)\rightarrow\mathcal{H}^{0,1} $,
\end{enumerate}
where $\mathcal{H}^{0,1}$ is the vector space of all harmonic vector $(0,1)$-form on $X_0$ (this space can also be canonically identified with $H^1(X_0,\Theta)$), such that for all $v\in A^{0,1}(\Theta) $, we have 
\begin{equation}
v=Hv+\square Gv.
\end{equation}
\begin{lem} The linear operators $G$ and $H$ are $K$-equivariant.

\end{lem}
\begin{proof}
For any $v \in \square A^{0,1}(\Theta) $ and $g\in K$, $gv$ is also in $\square A^{0,1}(\Theta)$ for the sake of $K$-invariance of $\square A^{0,1}(\Theta)$. Thus, by (4.2) we have that $$v=\square Gv \text{ and }  gv=\square Ggv.$$ So, the $K$-equivariance of $\square$ gives us
$$\square \left( g  Gv\right)=g \square \left( Gv \right)=gv.$$ Hence, 
$$\square \left( Ggv -gGv \right) =0$$ so that $ Ggv -gGv \in \mathcal{H}^{0,1}$. On the other hand, $gGv \in \square A^{0,1}(\Theta)$, and so is $Ggv -gGv. $ Consequently,
$$Ggv -gGv \in \mathcal{H}^{0,1} \cap\square A^{0,1}(\Theta) =\left \{ 0 \right \} $$ so that $$Ggv =gGv$$ for any $v \in \square  A^{0,1}(\Theta)$ and $g \in G$.

Now for any $v \in  A^{0,1}(\Theta) $ and $g\in K$, we have that \begin{align*}
gGv &= gG(Hv+\square Gv)\\ 
 &=gGHv+gG\left ( \square Gv \right )\\ 
 &=gG\left ( \square Gv \right )\text{ since }GH=0, \\ 
 &=G\left ( g\square Gv \right ) \text{ by the above case},  \\ 
 &= G g\left (v-Hv\right ) \text{ by the decomposition (4.2)},\\
&=Ggv-GgHv \\
&=Ggv \text{ since }\mathcal{H}^{0,1}\text{ is also } K\text{-invariant}.
\end{align*} Thus, the $K$-equivariance of $G$ follows.

For the $K$-equivariance of $H$, we have that
\begin{align*}
gHv &=g(v-\square Gv) \\ 
 &= gv-g\square Gv\\
 &= gv-\square Ggv \text{ since }\square, G \text{ are } K\text{-equivariant},\\ 
 &= Hgv.
\end{align*}
This ends the lemma.
\end{proof}
Next, Kuranishi would like to parametrize the set
$$\Phi:=\left \{  \phi \in A^{0,1}(\Theta) \mid \overline{\partial}\phi-\frac{1}{2}[\phi,\phi]=0, \overline{\partial}^*\phi=0\right \}$$ which actually forms an effective and complete family. We shall repeat briefly his argument. For any $\phi \in \Phi$, we have that
$$\square \phi -\frac{1}{2}\overline{\partial}^*[\phi,\phi]=0.$$ Applying Green's operator on this, we get
$$\phi -\frac{1}{2}G\overline{\partial}^*[\phi,\phi]=H\phi.$$ Thus, $\Phi$ is a subset of $$\Psi:=\left \{  \phi \in A^{0,1}(\Theta) \mid \phi-\frac{1}{2}G\overline{\partial}^*[\phi,\phi]\in \mathcal{H}^{0,1}\right \}. $$
Therefore, it is natural to parametrize $\Psi$ first. Let $\lbrace U_\sigma\rbrace$ be a finite covering of $X_0$ and $x_\sigma=(x_{1\sigma},\cdots,x_{n\sigma})$ be a local chart of $X_0$ on $U_\sigma$.  Let $\lbrace f_\sigma \rbrace$ be a smooth partition of unity with respect to the covering $\lbrace U_\sigma\rbrace$ of $X_0$. We introduce another norm in $A^{0,1}(\Theta)$. For $l=(l_1,\cdots,l_n)$, where $l_j$ is some non-negative integer ($j\in [1,\cdots,n]$), we denote by $D_\sigma^l$, the partial derivative
$$\left ( \frac{\partial }{\partial x_{1_\sigma}} \right )^{l_1}\cdots\left ( \frac{\partial }{\partial x_{n_\sigma}} \right )^{l_n}$$ and set $\left | l \right |=l_1+\cdots+l_n$. For $u \in A^{0,1}(\Theta)$ and for an integer $k\geq 0$, we set $$\left \| u \right \|_k^{2}=\sum_\sigma\sum_{\left | l \right |\leq k}\int \left \langle D_\sigma^lf_\sigma u(x_\sigma) ,D_\sigma^lf_\sigma u(x_\sigma)\right \rangle dv $$
where $dv$ is the volume element of $X_0$. This norm is called Sobolev $k$-norm. From now on, we fix once for all a sufficiently large integer $k$. Let $\mathfrak{H}^k(\Theta)$ be the Hilbert space obtained by completing $ A^{0,1}(\Theta)$ with respect to this Sobolev $k$-norm. Making use of Inverse Mapping Theorem for Banach manifolds to the map \begin{align*} 
F :A^{0,1}(\Theta)&\rightarrow A^{0,1}(\Theta)\\
\phi &\mapsto \phi-\frac{1}{2}G\overline{\partial}^*[\phi,\phi],
\end{align*} 
there exists a complex Banach analytic map
$\phi: W \rightarrow \mathfrak{H}^k(\Theta)$ such that 
$$s=F\phi(s)=\phi(s)-\frac{1}{2}G\overline{\partial}^*[\phi(s),\phi(s)]$$ for all $s \in W$, where $$W:= \left \{  s\in \mathcal{H}^{0,1} \mid  \left \| s \right \|_k< \epsilon \right \}  $$ and $\epsilon$ is sufficiently small. Hence, for $s \in W$, we have that $$\square \phi(s)-\frac{1}{2}\overline{\partial}^*[\phi(s),\phi(s)]=0,$$ which follows from the fact that $\square G\overline{\partial}^*= \overline{\partial}^*$ and that $s$ is harmonic. By the regularity of elliptic differential operators, we deduce that $\phi$ is holomorphic and that the image of $\phi$ is actually in $A^{0,1}(\Theta)$. In other words, we obtain a holomorphic map 
\begin{equation}
\phi: W \rightarrow A^{0,1}(\Theta)
\end{equation}
 whose image, by construction, covers a neighborhood of $0$ in $\Psi$ and so, a neighborhood of $0$ in $\Phi$.

Finally, a necessary and sufficient condition on $s$ for $\phi(s)$ to be in $\Phi$ is that $H[\phi(s),\phi(s)]=0$. Set $S':= \left \{  s\in W \mid  H[\phi(s),\phi(s)]=0\right \} $. Restricting on $S'$, we obtain a holomorphic map 
\begin{equation}
\phi: S' \rightarrow A^{0,1}(\Theta)
\end{equation} which satisfies the conditions ($i$), ($ii$), ($iii$) and ($iv$) in Theorem 3.1.

Now, we add the $K$-action. Recall that the Lie bracket $[\cdot,\cdot]$ on $A^{0,1}(\Theta)$ is defined as follows. For two element $\alpha,\beta \in A^{0,1}(\Theta)$ given in local coordinates 
$$\alpha =\sum m_i^ud\overline{z}^i\bigotimes \frac{\partial}{\partial z^u} \text{ and } \beta =\sum n_j^vd\overline{z}^j\bigotimes \frac{\partial}{\partial z^v}$$
then $$\left [ \alpha,\beta \right ]: =\sum d\overline{z}^i\wedge d\overline{z}^j\bigotimes\left [m_i^u\frac{\partial}{\partial z^u}, n_j^v\frac{\partial}{\partial z^v} \right ]'$$ where $\left[\cdot,\cdot \right ]'$ is the usual Lie bracket for the Lie algebra of vector fields on $X_0$. Let $g \in K$ then
$$g.\alpha := \sum g^*\left (d\overline{z}^i  \right )\bigotimes g_*\left (g_i^u\frac{\partial}{\partial z^u}  \right ) $$ where $g^*$ and $g_*$ are the pull-back of differential forms and the push-forward of vector fields, respectively. With this definition, the $G$-action clearly commutes with the Lie bracket, i.e.
$$g[\cdot,\cdot]=[g\cdot,g\cdot]$$ because the wedge product $\wedge$ and the Lie bracket $\left[\cdot,\cdot \right ]'$ do. Moreover, $G$ and $\overline{\partial}^*$ are $K$-equivariant. Thus, $F$ is also $K$-equivariant. 
\begin{lem} There exists an open neighborhood $U$ of $0$ contained in $W$ such that $U$ is $K$-invariant.

\end{lem}
\begin{proof}
For each $g\in K$, there exists a neighborhood $V_g$ of $g$ and $K_g$ of $0$ such that $V_g.K_g \in W$. By the compactness of $K$, there exists a finite set $I \subset K$ such that $K=\bigcup_{g\in I}V_g$. Let $P=\bigcap_{g\in I}K_g$ then $P$ is an open neighborhood of $0$ in $\mathcal{H}^{0,1}$. Thus,
$$K.P =\left (\bigcup_{g\in I}V_g   \right ). \left (\bigcap_{g\in I}K_g  \right ) \subseteq W.$$
Finally, set $U:=\bigcup_{g\in K}V_gK$. This is the desired $K$-invariant open neighborhood of $0$ contained in $W$.
\end{proof}
Now, restricting the map in (4.3) on this $U$, we obtain a map 
\begin{equation}
\phi: U \rightarrow \phi(U)\subseteq A^{0,1}(\Theta)
\end{equation} which is $K$-equivariant because it is the inverse of the $K$-equivariant map $F$ on $U$. Finally, set $S:= S'\cap U$.
\begin{lem}
$S$ is $K$-invariant and this $K$-action is real analytic.
\end{lem}
\begin{proof}
Let $s \in S'\cap U$ and $g \in K$ then we have
\begin{align*}
H[\phi(g.s),\phi(g.s)]&=H[g.\phi(s),g.\phi(s)] \text{ since } \phi \text{ is } K\text{-equivariant on }U,\\ 
 &=Hg.[\phi(s),\phi(s)] \text{ since the action commutes with the bracket},\\ 
 &=gH [\phi(s),\phi(s)] \text { by Lemma 4.1},\\ 
 &= g.0 \text { since s \ in }S',\\ 
 &= 0.
\end{align*} Thus, $g.s \in S'$. Moreover, $g.s \in U$ by the construction of $U$. Hence, $g.s \in S'\cap U $ so that $S$ is $K$-invariant. The part that this $K$-action on $S$ is real analytic follows from the fact that it is the restriction of a linear $K$-action on $U$.
\end{proof}
\begin{proof}[Proof of Theorem 4.1]
The restriction of the map $\phi$ in (4.5) on $S$ gives us a map 
$$
\phi: S \rightarrow  A^{0,1}(\Theta)
$$ which satisfies all the conditions given in the theorem.
\end{proof}

\section{The case that $G$ is a complex reductive Lie group}
In this final section, we would like to extend Corollary 4.1 to the case that $G$ is a complex reductive Lie group. 

We begin by introducing the definition of holomorphic local $(G,K)$-action on a complex space $X$ where $K$ is a compact subgroup of $G$. Denote by $\prod_X$ the collection of all pair $\pi=(U_\pi,V_\pi)$, where $U_\pi$ and $V_\pi$ are open subsets in $X$ such that $U_\pi\Subset V_\pi$. Suppose that for each $\pi \in \prod_X$ we have an open neighborhood $G_\pi$ of $K$ and a mapping $\Phi_\pi:\; G_\pi \rightarrow \Ho(U_\pi,V_\pi)$ where $\Ho(U_\pi,V_\pi)$ is the set of all holomorphic functions from $U_\pi$ to $V_\pi$.
\begin{defi} One says that the system $\lbrace \Phi_\pi\rbrace$ defines a local $(G,K)$-action on $X$ if the following conditions are satisfied.
\begin{enumerate}
\item[(a)] For all $g,h \in G$ such that $k:=gh\in G_\pi$, we have
$$\Phi_{\pi}(g)\circ \Phi_\pi(h)\mid_{U_{\pi,h}}= \Phi_\pi(k)\mid_{U_{\pi,h}}$$ where $U_{\pi,h}:=\lbrace x \in U_\pi \mid \Phi_\pi(h)(x)\in U_\pi \rbrace $;
\item[(b)] $\Phi_\pi(\mathbf{1}_G)=\mathbf{id}$;
\item[(c)] for all $\pi, \rho \in \prod_X$ and $g\in G_\pi\cap G_\rho$ we have
$$\Phi_\pi(g)\mid_{U_{\pi}\cap U_{\rho}}=\Phi_\rho(g)\mid_{U_{\pi}\cap U_{\rho}} $$ so that $gx:=\Phi_\pi(g)x$ is independent of the choice of $\pi$ with $x\in U_\pi,g\in G_\pi$;
\item[(d)] for any two open sets $U\Subset U_\pi$ and $V\Subset V_\pi$, the set 
$$W:=W_{\overline{U},V}:=\lbrace g\in G_\pi \mid g\cdot \overline{U}\subset V \rbrace $$ is open in $G_\pi$ and the map \begin{align*}
*:W&\rightarrow \mathcal{O}(U)\\
g &\mapsto f\circ g\mid_U
\end{align*} is continuous for all $f\in \mathcal{O}(V)$ where $\overline{U}$ is the closure of $U$ and $\mathcal{O}(P)$ is the set of holomorphic functions on $P$ for any open subset $P$ of $X$;
\item[(e)] The restriction of the system $\lbrace \Phi_\pi\rbrace$ on $K$ gives a global $K$-action on $X$, i.e. a homomorphism of topological groups $\Phi:\; K\rightarrow \Au(X)$.
\end{enumerate}
Moreover, if $G$ is a real (resp. complex) Lie group and if  $*$ and $\Phi$ are real analytic (resp. holomorphic), then the local $(G,K)$-action is called real analytic (resp. holomorphic). Two local $(G,K)$-actions defined by two systems $\lbrace \Phi_\pi\rbrace$ and $\lbrace \Phi'_\pi\rbrace$  are said to be equivalent if for all $\pi\in \prod_X$, the mappings $\Phi_\pi:\; G_\pi \rightarrow \Ho(U_\pi,V_\pi)$ and $\Phi'_\pi:\; G'_\pi \rightarrow \Ho(U_\pi,V_\pi)$ coincide on a sub-domain $G_\pi \cap G'_\pi$ containing $K$ and their restrictions on $K$ give the same global $K$-action. 
\end{defi}

As before, by local $G$-action, we really mean real analytic local $G$-action. If we let $K$ be the identity element of $G$ in Definition 5.1 then we recover the usual definition of (holomorphic) local $G$-action on complex spaces (see [1], Section 1.2 for more details). In this case, we have the following theorem ([1], page 25, Corollary).
\begin{thm}
Let $G$ be a (complex) Lie group, $\mathfrak{g}$ the Lie algebra of $G$, and $S$ a complex space. Then we have two bijections
$$\begin{Bmatrix}
equivalence\; classes \; of \\  local \;
G\text{-}actions\;  on \; S

\end{Bmatrix} \longleftrightarrow \begin{Bmatrix}
 Lie \;algebra \;homomorphisms \\ \mathfrak{g} \rightarrow \mathcal{T}_S(S)

\end{Bmatrix}$$

$$\begin{Bmatrix}
equivalence\; classes \; of \\ holomorphic\;  local \\
G\text{-}actions\;  on \; S

\end{Bmatrix} \longleftrightarrow \begin{Bmatrix}
complex \;Lie \;algebra \\homomorphisms \;\mathfrak{g} \rightarrow \mathcal{T}_S(S)

\end{Bmatrix}$$
where $\mathcal{T}_S(S)$ is the set of holomorphic vector fields on $S$.
\end{thm}

\begin{coro}  Let $K$ be a connected compact real Lie group acting on a complex space $X$  and $G$ be the complexification of $K$. There exists a holomorphic local $(G,K)$-action on $X$ extending the initial global $K$-action.
\end{coro}
\begin{proof} By Theorem 5.1, the initial $K$-action gives us a Lie algebra homomorphism $\varphi:$ $\Lie(K) \rightarrow \mathcal{T}_X(X)$. Since $\mathcal{T}_X(X)$ is a complex Lie algebra, the $\mathbb{C}$-linear extension of $\varphi$ gives us a complex Lie algebra homomorphism $\varphi^{\mathbb{C}}:$ $\Lie(K)^{\mathbb{C}}=\Lie(G) \rightarrow \mathcal{T}_X(X)$. An application of Theorem 5.1 again provides a holomorphic local $G$-action on $X$. Note that the restriction of this holomorphic local $G$-action on $K$ gives a local $K$-action on $X$, which in fact is equivalent to the initial global one on $X$. This follows from the fact that they correspond to the same Lie algebra homomorphism $\varphi:$ $\Lie(K) \rightarrow \mathcal{T}_X(X)$. Thus, it allows us to define a holomorphic local $(G,K)$-action on $X$ as follows. If $g\in K$ then then the action of $g$ is determined by the initial global $K$-action. If $g\in G \setminus K$ then the action of $g$ is determined by the extended holomorphic local $G$-action. This ends the proof.
\end{proof}

The two following lemmas are helpful in the sequel.
\begin{lem} Let $f:$ $X \rightarrow Y$ be a proper surjective flat map of complex spaces whose geometric fibers are all connected complex compact manifolds. Then the natural maps $\mathcal{O}_Y \rightarrow f_*\mathcal{O}_X$ is an isomorphism.
\end{lem}
\begin{proof}
For $y \in Y$, we have that $H^0(X_y,\mathcal{O}_y)=\mathbb{C}(y)$ since $X_y$ is a compact complex manifold. So, the base change morphism 
$$\phi^0(s):\; (f_*\mathcal{O}_X)_x\otimes_{\mathcal{O}_{Y,y}}\mathbb{C}(y) \rightarrow H^0(X_y,\mathcal{O}_y)= \mathbb{C}(y)
$$ is clearly surjective. By [2, Chapter III, Theorem 3.4], $\phi^0(s)$ is an isomorphism. Note that $\phi^{-1}(s)$ is trivially surjective. So, an easy application of [2, Chapter III, Corollary 3.7] gives us the freeness of the $\mathcal{O}_Y$-module $f_*\mathcal{O}_X$ in a neighborhood of $y$. As $\phi^0(y)$ is an isomorphism then $f_*\mathcal{O}_X$ is free of rank $1$ in a neighborhood of $y$. But this holds for any $y \in Y$. Thus, $f_*\mathcal{O}_X$ is locally free of rank $1$ and then the map  $\mathcal{O}_Y \rightarrow f_*\mathcal{O}_X$ turns out to be an isomorphism. This completes the proof.\end{proof}

\begin{lem} Let $G$ be a complex reductive group and let $K$ be a connected real maximal compact subgroup such that $K^{\mathbb{C}}=G$.
Let $Q$ be open subset of $G$. Let $g$ be a point in $G$ such that the $K$-orbit $K.g$ intersects every connected component of $Q$. Then if $f$ is a holomorphic function on $Q$ such that $f\mid_{K.g \cap Q}=0$ then $f\equiv 0$ on $Q$.
\end{lem}
\begin{proof}
See ([6], page 634, Identity Theorem).
\end{proof}
Now, we are ready to state the second main result of this paper.
\begin{thm}
 Let $X/S$ be the Kuranishi family of a complex compact manifold $X_0$ with a holomorphic action of a complex reductive Lie group $G$. Then we can provide holomorphic local $G$-actions on $X/S$ extending the holomorphic $G$-action on $X_0$. 
\end{thm}
\begin{proof}
Let $K$ be a connected real maximal compact subgroup whose complexification is exactly $G$. By Corollary 4.1, we obtain a $K$-equivariant Kuranishi family $\pi:$ $X\rightarrow S$. If we can extend the $K$-actions on $X$ and on $S$ to holomorphic local $(G,K)$-actions such that $\pi$ is $G$-equivariant with respect to these holomorphic local $(G,K)$-actions then our result follows immediately since any local $(G,K)$-action is obviously a local $G$-action. By Corollary 5.1, we obtain a holomorphic local $(G,K)$-action on $X$. Note that the restriction on $K$ of this local $(G,K)$-action is nothing but the initial global $K$-action on $X$. 

Let $g\in N(K)\setminus K$ where $N(K)$ is a neighborhood of $K$. We shall prove that $g$, as a biholomorphism on $X$, swaps fibers of $\pi$. Indeed, recall that by construction, $S$ is an analytic subset defined in a open subset $U\subset \mathbb{C}^n$ where $n:=\dim_{\mathbb{C}}H^1(X_0,\Theta)$. Consider the following holomorphic function 
$$\rho_i:\; X\overset{g}{\rightarrow}X\overset{\pi}{\rightarrow}S\overset{\iota}{\rightarrow}\mathbb{C}^n\overset{\pi_i}{\rightarrow}\mathbb{C}$$ where $\iota$ is the inclusion and $\pi_i$ is the $i^{\text{th}}$-projection.  Lemma 5.1 tells us that $\pi_*\mathcal{O}_X=\mathcal{O}_S$ which means precisely that any holomorphic function from $X$ to $\mathbb{C}$ factors through $\pi$. So, for each $i$, there exists a holomorphic function $\sigma_i:\; S\rightarrow \mathbb{C}$ such that $\rho_i=\sigma_i\circ \pi$. So, $\sigma_i$'s together form a holomorphic function $\sigma: \; S \rightarrow \mathbb{C}^n$ which then is lifted to a holomorphic function $\nu_g:$ $S\rightarrow S$. More precisely, we have the following commutative diagram
\begin{center}
\begin{tikzpicture}[every node/.style={midway}]
  \matrix[column sep={8em,between origins}, row sep={3em}] at (0,0) {
    \node(Y){$X$}   ; & \node(X) {$X$} ;   \\
    \node(M) {$S$}; & \node (N) {$S$};\\
    \node(P) { }; & \node (Q) {$\mathbb{C}^n$};\\
  };
  
  \draw[->] (Y) -- (M) node[anchor=east]  {$\pi$}  ;
  \draw[->] (Y) -- (X) node[anchor=south]  {$g$};
  \draw[->] (X) -- (N) node[anchor=west] {$\pi$};
  \draw[->] (N) -- (Q) node[anchor=west] {$\iota$};
  \draw[->] (M) -- (Q) node[anchor=north] {$\sigma$};
  \draw[->] (M) -- (N) node[anchor=south] {$\nu_g$};,
\end{tikzpicture}
\end{center}
which means in particular that $g$ exchanges fibers of $\pi$. Since $g$ is a biholomorphism then so is $\nu_g$. On one hand, $\nu_g$ is uniquely determined by $g$. This follows from the fact that $X$ is constructed from $\underline{X}\times S$, as the underlying differentiable manifold, and the fact that  $g$ swaps fibers of $\pi$. On the other hand, since the local $(G,K)$-action on $X$ is holomorphic then $\nu_g(-)$ varies holomorphically with respect to the variable $g$. Hence, the map $g \mapsto \nu_g$ defines a holomorphic local $(G,K)$-action on $S$, which extends the initial $K$-action on $S$.

Finally, we shall prove that the restriction of the holomorphic local $(G,K)$-action of $X$ on the central fiber $X_0$ is the initial $G$-action on $X_0$. In order to do it, we first show that the holomorphic local $(G,K)$-action on $S$ fixes the reference point $0$. Let $N(K)$ be a connected open neighborhood of $K$. Note that the holomorphic function  \begin{align*}
\chi :G&\rightarrow (S,0)\\
g &\mapsto \nu_g(0)
\end{align*} is constant on $K$, i.e. $\chi(k)=0$ for all $k \in K$. Consider the holomorphic function 
$$\mu_i:\; G\overset{\chi}{\rightarrow}(S,0)\overset{\iota}{\rightarrow}(\mathbb{C}^n,0)\overset{\pi_i}{\rightarrow}\mathbb{C}$$ where $\iota$ is the inclusion and $\pi_i$ is the $i^{\text{th}}$-projection. Hence, we also have $\mu_i(k)=0$ for all $k \in K$. Applying Lemma 5.2 with $g=\mathbf{1}_G$ and $Q=N(K)$, we obtain that $\mu_i$ is zero on $N(K)$. But this holds for any $i$ and so $\chi(g)=0$ for all $g \in N(K)$. This justifies the claim. Therefore, the local $(G,K)$-action on $X$ preserves the central fiber $X_0$, i.e. $gX_0 \subset X_0$ for $g\in G$ whenever it is defined. Consequently, we have a holomorphic local $(G,K)$-action on $X_0$, which is the restriction on $X_0$ of the one on $X$. Because $X_0$ is compact then this action turns out to be global and it contains the initial $K$-action on $X_0$. As a matter of fact, it must coincide with the initial $G$-action on $X_0$ because the action of $G$ on a complex compact manifold is uniquely determined by the one of $K$. 

In summary, what we have just done is to equip holomorphic local $(G,K)$-actions on $X$ and on $S$ in a way that the map $\pi:$ $X \rightarrow S$ is $G$-equivariant with respect two these holomorphic local $(G,K)$-actions and that the restriction on the central fiber $X_0$ of the holomorphic local $(G,K)$-action on $X$  is nothing but the initial holomorphic $G$-action on $X_0$. This finishes the proof. 
\end{proof}

\bibliographystyle{amsplain}

\end{document}